\documentclass[12pt,a4paper]{amsart}

\usepackage{amsfonts, amsmath, amssymb, amsthm, amscd, hyperref}

\newtheorem{thm}{Theorem}
\newtheorem*{thm*}{Theorem}
\newtheorem{lem}{Lemma}

\newtheorem{cor}[thm]{Corollary}

\theoremstyle{definition}
\newtheorem{defn}{Definition}

\theoremstyle{remark}

\DeclareMathOperator{\hind}{ind}

\DeclareMathOperator{\Sqe}{Sq_e}

\DeclareMathOperator{\tr}{tr}

\renewcommand{\epsilon}{\varepsilon}

\begin{document}

\title{$Z_2$-index of the grassmanian $G_{2n}^n$}

\author{R.N.~Karasev}
\thanks{This research is supported by the Dynasty Foundation, the President's of Russian Federation grant MK-113.2010.1, the Russian Foundation for Basic Research grants 10-01-00096 and 10-01-00139}

\email{r\_n\_karasev@mail.ru}
\address{
Roman Karasev, Dept. of Mathematics, Moscow Institute of Physics
and Technology, Institutskiy per. 9, Dolgoprudny, Russia 141700}



\begin{abstract}
We study the real Grassmann manifold $G_{2n}^n$ (of $n$-subspaces in $\mathbb R^{2n}$), and the action of $Z_2$ on it by taking the orthogonal complement. The homological index of this action is estimated from above and from below. In case $n$ is a power of two it is shown that $\hind G_{2n}^n=2n-1$.
\end{abstract}

\maketitle

\section{Introduction}

The topology of real Grassmannians has many applications in the discrete and convex geometry. For example, the Schubert calculus and other topological facts (e.g. from~\cite{cf1960,hil1980A}) can be applied to obtain some existence theorems for flat transversals (affine flats intersecting all members of a given family of sets), see~\cite{dol1993,ziv2004,kar2008bu,mk2010} for example.

In this paper we consider the Grassmannian $G_{2n}^n$ of $n$-dimensional subspaces of $\mathbb R^{2n}$. This space has a natural $Z_2$-action (involution) by taking the orthogonal complement of the subspace. The well-known invariant of $Z_2$-spaces is homological index, introduced and studied in~\cite{kr1952,schw1957,cf1960}, see also the book~\cite{mat2003} for a simplified introduction to the index and its many applications to combinatorics and geometry.

The following theorem gives an estimate for the index of the Grassmannian.

\begin{thm}
\label{index-est}
If $n=2^l(2m+1)$, then 
$$
2^{l+1}-1\le \hind G_{2n}^n \le 2n-1,
$$
for $n=2m+1$ the index equals $1$, for $n=2(2m+1)$ the index equals $3$.
\end{thm}

The lower and the upper bounds coincide for $n=2^l$, odd $n$, $n=2(2m+1)$. In other cases there is still some gap between them. This result easily produces some geometric consequences. Here is one example (it also uses Lemma~\ref{index-bu} below).

\begin{cor}
Let $n=2^l(2m+1)$, $k=2^{l+1}-1$. Consider some $k$ continuous (in the Hausdorff metric) $O(n)$-invariant functions $\alpha_1,\ldots,\alpha_k$ on (convex) compacts in $\mathbb R^n$. Then for any (convex) compact $K\subseteq \mathbb R^{2n}$ there exist a pair of orthogonal $n$-dimensional subspaces $L$ and $M$, such that for their respective orthogonal projections $\pi_L$ and $\pi_M$ we have
$$
\forall i=1,\ldots,k\ \alpha_i(\pi_L(K)) = \alpha_i(\pi_M(K)).
$$
\end{cor}

In this corollary $\alpha_i$ can be the Steiner measures (volume, the boundary measure, the mean width, etc.) for example. The same statement holds if we consider a point $x\in K$ and sections of $K$ by mutually orthogonal affine $n$-subspaces $L$ and $M$ through $x$, instead of projections to $L$ and $M$.

The author thanks O.R.~Musin for drawing attention to the problem of calculating the $Z_2$-index of $G_{2n}^n$ and for the discussion.

\section{Preliminary observations}

Let us state some topological definitions on spaces with group action, see~\cite{hsiang1975} for more detailed discussion.

\begin{defn}
Let $G$ be a compact Lie group or a finite group. A space $X$ with continuous action of $G$ is called a \emph{$G$-space}. A continuous map of $G$-spaces, commuting with the action of $G$ is called a \emph{$G$-map} or an \emph{equivariant map}. A $G$-space is called \emph{free} if the action of $G$ is free.
\end{defn}

There exists the universal free $G$-space $EG$ such that any other $G$-space maps uniquely (up to $G$-homotopy) to $EG$. The space $EG$ is homotopy trivial, the quotient space is denoted $BG = EG/G$. For any $G$-space $X$ and an Abelian group $A$ the equivariant cohomology $H_G^*(X, A)=H^*(X\times_G EG, A)$ is defined, and for free $G$-spaces the equality $H_G^*(X, A) = H^*(X/G, A)$ holds.

Consider the case $G=Z_2$. Note that 
$$
H_G^*(\mathrm{pt}, Z_2) = H^*(\mathbb RP^\infty, Z_2) = Z_2[c]=\Lambda,
$$ 
where the dimension of the generator is $\dim c = 1$. Since any $G$-space $X$ can be mapped to the point $\pi_X: X\to \mathrm{pt}$, we have a natural map $\pi_X^* : \Lambda\to H_G^*(X, Z_2)$, the image $c$ under this map will be denoted by $c$, if it does not make a confusion. The generator element of $Z_2$ will be denoted by $\sigma$.

\begin{defn}
The \emph{cohomology index} of a $Z_2$-space $X$ is the maximal $n$ such that the power $c^n\not=0$ in $H_G^*(X,Z_2)$. If there is no maximum, we consider the index equal to $\infty$. Denote the index of $X$ by $\hind X$.
\end{defn}

Let us state the following well-known lemma.

\begin{lem}[The generalized Borsuk-Ulam theorem for odd maps]
\label{index-bu}
If there exists an equivariant map $f: X\to Y$, then $\hind X\le \hind Y$.
\end{lem}

Now we are ready to prove the upper bound in Theorem~\ref{index-est}.

\begin{lem}
\label{index-est-up}
$$
\hind G_{2n}^n \le 2n-1.
$$
\end{lem}

\begin{proof}
Let us parameterize $G_{2n}^n$ by the orthogonal projection matrices $P$. These matrices are characterized by the equations
$$
P^t=P,\ P^2=P,\ \tr P = n.
$$
The action of $Z_2$ is given by ($E$ is the unit matrix)
$$
\sigma(P) = E-P.
$$
Now consider the map $f: G_{2n}^n\to \mathbb R^{2n}$, defined by the coordinates
$$
f_1(P) = P_{11}-1/2,\ f_i(P)=P_{1i}, (i=2,\ldots, 2n).
$$
This map is $Z_2$-equivariant, if the action on $\mathbb R^{2n}$ is antipodal, i.e. $\sigma: x\mapsto -x$. Note also that $f(P)$ is never zero, otherwise $P$ would have an eigenvalue $1/2$, which is not true. Hence $f$ composed with the projection $\mathbb R^{2n}\setminus\{0\}\to S^{2n-1}$ gives the equivariant map 
$$
\tilde f : G_{2n}^n\to S^{2n-1},
$$
and the result follows by Lemma~\ref{index-bu}.
\end{proof}

\begin{lem}
\label{index-div}
Suppose $n=ds$ for some positive integers $d,s$. Then
$$
\hind G_{2n}^n \ge \hind G_{2d}^d.
$$ 
\end{lem}

\begin{proof}
Let us decompose 
$$
\mathbb R^{2n} = \mathbb R^{2d}\oplus\dots\oplus \mathbb R^{2d}
$$
into $s$ summands. Consider a $d$-subspace $L\in G_{2d}^d$, and define with the above decomposition
$$
f(L) = L\oplus\dots\oplus L\subset \mathbb R^{2n}.
$$
The map $f:G_{2d}^d\to G_{2n}^n$ is evidently equivariant and by Lemma~\ref{index-bu} we obtain the inequality.
\end{proof}

In order to prove Theorem~\ref{index-est} it remains to prove the following lemmas.

\begin{lem}
\label{index-odd}
If $n$ is odd, then $\hind G_{2n}^n=1$, if $n=2\mod 4$, then $\hind G_{2n}^n=3$.
\end{lem}

\begin{lem}
\label{index-2l}
If $n=2^l$, then $\hind G_{2n}^n = 2n-1$.
\end{lem}

\section{External Steenrod squares}
\label{ext-sq-sec}

In order to prove Lemma~\ref{index-2l}, we have to describe the cohomology of the subgroup $G\subset O(2n)$, generated by the subgroup $O(n)\times O(n)$ (from some decomposition $\mathbb R^{2n}=\mathbb R^n\oplus \mathbb R^n$), and $Z_2$ that interchanges the summands $\mathbb R^n$. This group is the wreath product $O(n)\wr Z_2 = (O(n)\times O(n))\rtimes Z_2$.

In order to describe the cohomology of a wreath product, we have to use the construction of external Steenrod squares. We mostly follow~\cite[Ch.~V]{brs1976}, where the Steenrod squares were defined in the unoriented cobordism. The cobordism was defined using mock bundles, if we allow the mock bundles to have codimension $2$ singularities, we obtain ordinary cohomology modulo $2$. In the sequel we consider the cohomology modulo $2$ and omit the coefficients in notation. This construction is known and was used in~\cite{hung1990} to describe the modulo $2$ cohomology of the symmetric group and configuration spaces. Still, for reader's convenience we give a short and self-contained explanation here.

The construction of the external Steenrod squares on a polyhedron $K$ starts with the fiber bundle (for some integer $n>0$)
$$
\sigma_K: (K\times K\times S^n)/Z_2\to S^n/Z_2=\mathbb RP^n.
$$
The group $Z_2$ acts by permuting $K\times K$, and antipodally on $S^n$.
Consider a cohomology class $\xi\in H^*(K)$, represented by a mock bundle $\xi : E(\xi)\to K$. Then the mock bundle 
$$
(\xi\times \xi\times S^n)/Z_2\to (K\times K\times S^n)/Z_2
$$
is the external Steenrod square $\Sqe \xi$. The operation $\Sqe$ is evidently multiplicative, in~\cite[Ch.~V, Proposition~3.3]{brs1976} it is claimed that $\Sqe$ is also additive. We are going to show that it is not true, first we need a definition.

\begin{defn}
The difference $\Sqe(\xi+\eta) - \Sqe\xi - \Sqe\eta$ is represented by the mock bundle
$$
\xi \odot \eta = (\xi\times\eta\times S^n + \eta\times\xi\times S^n)/Z_2,
$$
where $Z_2$ exchanges the components $\xi\times\eta$ and $\eta\times\xi$. 
\end{defn}

Since the fiber of $\sigma_K$ is $K\times K$, the restriction of $\xi\odot \eta$ to the fiber is $\xi\times\eta+\eta\times\xi$, which is nonzero if $\eta\not=\xi$ as cohomology classes. Thus the operation $\odot$ is not trivial.

We need a lemma about the $\odot$-multiplication.

\begin{lem}
\label{odot-ann}
Denote $c$ the hyperplane class in $H^1(\mathbb RP^n)$. Then for any $\xi,\eta\in H^*(K)$ the product 
$$
(\xi\odot\eta)\smile\sigma_K^*(c) = 0
$$
in $H^*((K\times K\times S^n)/Z_2)$.
\end{lem}

\begin{proof}
Consider the mock bundle 
$$
\alpha = \xi\times\eta\times S^{n-1} + \eta\times\xi\times S^{n-1},
$$
which has the natural $Z_2$-action, it represents $(\xi\odot\eta)\smile\sigma_K^*(c)$ after taking the quotient by the $Z_2$-action.

Now divide $S^n$ into the upper and the lower half-spheres $H^+$ and $H^-$. Consider the mock bundle (with boundary) 
$$
\beta = \xi\times\eta\times H^+ + \eta\times\xi\times H^-
$$
over $K\times K\times S^n$. The action of $Z_2$ on $\beta$ is defined by permuting the summands and the antipodal identification of $H^+$ and $H^-$. Now it is clear that $\alpha$ is the boundary of $\beta$, and $\alpha/Z_2$ is the boundary of $\beta/Z_2$. Hence it is zero in the cohomology, and the similar statement is true for the unoriented bordism.
\end{proof}

We have to introduce another operation.

\begin{defn}
Let $\xi :E(\xi)\to K$, $\eta : E(\eta)\to K$ be two mock bundles. Let $p_+,p_-$ be the north and the south poles of $S^n$. Denote the mock bundle over $(K\times K\times S^n)/Z_2$
$$
\iota(\xi\times\eta) = (\xi\times\eta\times\{p_+\} + \eta\times\xi\times\{p_-\})/Z_2.
$$
\end{defn}

It is obvious from the definition that we have relation
$$
\iota(\xi\times\eta)\smile \sigma_K^*(c) = 0,
$$
it is also obvious that 
$$
\iota(\xi\times\xi) = \Sqe\xi\smile \sigma_K^*(c)^n.
$$
Let us describe the $\smile$-multiplication of the Steenrod squares, $\odot$, and $\iota(\ldots)$ classes. The following formulas are obvious from the definition:
$$ 
(\xi\odot\eta)\smile(\zeta\odot\chi) = (\xi\smile\zeta)\odot(\eta\smile\chi) + (\xi\smile\chi)\odot(\eta\smile\zeta),
$$
$$
(\xi\odot\eta)\smile(\Sqe\zeta) = (\xi\smile\zeta)\odot(\eta\smile\zeta),
$$
$$
(\xi\odot\eta)\smile\iota(\zeta\odot\chi) = \iota((\xi\smile\zeta)\times(\eta\smile\chi)) + \iota((\xi\smile\chi)\times(\eta\smile\zeta)),
$$
$$
\Sqe\xi\smile\Sqe\eta = \Sqe(\xi\smile\eta),
$$
$$
\Sqe\xi\smile \iota(\eta\times\zeta) = \iota((\xi\smile\eta)\times(\xi\smile\zeta)),
$$
$$
\iota(\xi\times\eta)\smile\iota(\zeta\times\chi)=0.
$$
Now we can describe the structure of the cohomology $H^*( (K\times K\times S^n)/Z_2)$.

\begin{defn}
Consider a $Z_2$-algebra $A$ with linear basis $v_1,\ldots, v_n$. Denote $A\odot A$ the subalgebra of $A\otimes A$, invariant w.r.t. $Z_2$-action by permutation. The linear base of $A$ is 
$$
\{v_i\otimes v_i\}_{i=1}^n,\ \{v_i\otimes v_j+v_j\otimes v_i\}_{i<j}. 
$$
\end{defn}

\begin{defn}
Consider a $Z_2$-algebra $A$ with linear basis $v_1,\ldots, v_n$. Denote $\iota(A\otimes A)$ the quotient vector space $A\otimes A/(v_i\otimes v_j + v_j\otimes v_i)$. As $Z_2$-algebra it has zero multiplication.
\end{defn}

\begin{lem}
\label{ext-sq}
The maps $\Sqe$, $\odot$, map the algebra $H^*(K)\odot H^*(K)$ to $H^*( (K\times K\times S^n)/Z_2)$. The map $\iota$ maps $\iota(H^*(K)\otimes H^*(K))$ to $H^*( (K\times K\times S^n)/Z_2)$. The images of these maps generate the cohomology $H^*( (K\times K\times S^n)/Z_2)$.

The latter cohomology can be described as the quotient of $H^*(K)\odot H^*(K)\otimes Z_2[c]\oplus \iota(H^*(K)\otimes H^*(K))$ by the relations
$$
c^{n+1}=0,\ (\xi\odot\eta)\otimes c = 0, \Sqe\xi\otimes c^n = \iota(\xi\otimes\xi).
$$ 
the $c$ is the preimage of the hyperplane class in $H^1(\mathbb RP^n)$.
\end{lem}

Note the important particular case: if $n\to\infty$, we image of $\iota(\ldots)$ disappears, and we also can take the quotient of $H^*(K)\odot H^*(K)$ by the linear span of all $\xi\odot\eta$ for $\xi,\eta\in H^*(K)$. Hence, the cohomology $H^*((K\times K\times S^\infty)/Z_2)$ has a quotient isomorphic to $\Sqe(H^*(K))\otimes Z_2[c]$. Here $\Sqe(H^*(K))$ is the same algebra as $H^*(K)$, but with twice larger degrees.  

\begin{proof}
The Leray-Serre spectral sequence for $\sigma_K$ starts with 
$$
E_2^{p,q} = H^p(\mathbb RP^n, \mathcal H^q(K\times K)).
$$
Let us describe the sheaf $\mathcal H^*(K\times K)$. If  $v_1,\ldots, v_n$ is the linear basis of $H^*(K)$, then an element $v_i\otimes v_i$ gives a subsheaf, isomorphic to the constant sheaf $Z_2$. The two elements $v_i\otimes v_j$ and $v_j\otimes v_i$ generate a non-constant sheaf $\mathcal A = Z_2\oplus Z_2$ with permutation action of $\pi_1(\mathbb RP^n)$. The cohomology $H^*(\mathbb RP^n,\mathcal A) = H^*(S^n, Z_2)$, since $\mathcal A$ is the direct image of $Z_2$ under the natural projection $\pi: S^n\to \mathbb RP^n$. Thus we know the structure of $E_2^{*,*}$.

The first column of $E_2$ is the $Z_2$-invariant elements of $H^*(K\times K)$, and all these elements are the restrictions of either $\Sqe\xi$ or $\xi\odot \eta$ to the fiber. Hence all the differentials of the spectral sequence are zero on the first column. The columns between the first and the last ($n$-th) are generated by multiplication with $c$, and the differentials are zero on them too. The last column is isomorphic to $\iota(H^*(K)\otimes H^*(K))$, the differentials are zero on it from the dimension considerations.

Hence in this spectral sequence $E_2=E_\infty$. Denote $v_1,\ldots v_n$ the linear base of $H^*(K)$. The first column of $E_2$ has the linear base 
$$
\{v_i\times v_i\}_{i=1}^n,\ \{v_i\times v_j+ v_j\times v_i\}_{i<j},
$$
the columns No. $j=1,2,\ldots,n-1$ have the linear base
$$
\{(v_i\times v_i)c^j\}_{i=1}^n,
$$ 
and the last column has the linear base
$$
\{\iota(v_i\times v_j)\}_{i,j=1}^n.
$$
From the definition of $\Sqe$, $\odot$, and $\iota(\ldots)$ the final cohomology $H^*( (K\times K\times S^n)/Z_2)$ is described the same way with $v_i\times v_i$ replaced by $\Sqe v_i$, and $v_i\times v_j+ v_j\times v_i$ replaced by $v_i\odot v_j$.
\end{proof}

Now consider a vector bundle $\nu : E(\nu)\to K$ and define 
$$
\Sqe\nu : (E(\nu)\times E(\nu)\times S^n)/Z_2\to (K\times K\times S^n)/Z_2.
$$
The Stiefel-Whitney classes of $\Sqe\nu$ are described by the following lemma.

\begin{lem}
\label{sw-sq}
Let $\dim \nu = k$, and let the Stiefel-Whitney class of $\nu$ be 
$$
w(\nu) = w_0+w_1+\ldots + w_k.
$$
Then
$$
w(\Sqe\nu) = \sum_{0\le i < j \le k} w_i\odot w_j + \sum_{i=0}^k (1+c)^{k-i} \Sqe w_i,
$$
where $c$ is the image of the hyperplane class in $H^1(\mathbb RP^n)$.
\end{lem}

\begin{proof}
Consider the case of one-dimensional $\nu$ first. Taking $n$ large enough we do not have to consider the image of $\iota(\ldots)$, then we can return to lesser $n$ by the natural inclusion 
$$
(K\times K\times S^n)/Z_2\to (K\times K\times S^{n+m})/Z_2. 
$$
The restriction of $\Sqe\nu$ to the fiber $K\times K$ has the Stiefel-Whitney class
$$
w(\nu\times\nu) = 1 + w_1(\nu)\times 1 + 1\times w_1(\nu) + w_1(\nu)\times w_1(\nu).
$$
Hence $w(\Sqe\nu)$ is either $1 + w_1(\nu)\odot 1 + \Sqe w_1(\nu)$, or $1 + w_1(\nu)\odot 1 + c + \Sqe w_1(\nu)$. Any point $x\in K$ gives a natural section
$$
s : S^n/Z_2\to (\{x\}\times\{x\}\times S^n)/Z_2
$$
of the bundle $\sigma_K$, and the bundle $s^*(\Sqe\nu)$ over $\mathbb RP^n$ is isomorphic to $\gamma\oplus\epsilon$, where $\gamma$ is the canonical bundle of the projective space, $\epsilon$ is the trivial bundle. Hence we should have 
$$
w(\Sqe\nu) = 1 + w_1(\nu)\odot 1 + c + \Sqe w_1(\nu).
$$

The general formula for $k>1$ follows from the splitting principle, suppose $\nu=\tau_1\oplus\dots\oplus\tau_k$, then
$$
w(\Sqe\nu) = \prod_{i=1}^k (1 + w_1(\tau_i)\odot 1 + c + \Sqe w_1(\tau_i)),
$$
and the result follows by removing parentheses.
\end{proof}

\section{The proof of Lemmas~\ref{index-odd} and \ref{index-2l}}

In order to calculate the index of $G_{2n}^n$, we describe the cohomology of $G_{2n}^n/Z_2$. Consider the subgroup $G=O(n)\wr Z_2$ of $O(2n)$, that is generated by two copies of $O(n)$ for some decomposition $\mathbb R^{2n}=\mathbb R^n\oplus\mathbb R^n$, and by the operator $\sigma$ that interchanges the summands of the decomposition. It is clear that $G_{2n}^n/Z_2 = O(2n)/G$.

The cohomology of $BO(n)$ is the polynomial algebra in Stiefel-Whitney classes
$$
H^*(BO(n)) = Z_2[w_1,\ldots, w_n].
$$
The group cohomology $H^*(BG)$ (by Lemma~\ref{ext-sq}) is generated by the external Steenrod squares $\Sqe w_1,\ldots, \Sqe w_n$, the generator $c\in H^1(BZ_2)$, and some combinations $x\odot y$ for $x,y\in H^*(BO(n))$, the relations are $(x\odot y)c=0$.

Let us find the kernel of the natural map $\pi^*: H^*(BG)\to H^*(O(2n)/G)$. The cohomology $H^*(O(2n)/G)$ can be calculated by considering the Leray-Serre spectral sequence with the term $E^{p,q}_2 = H^p(BG, \mathcal H^q(O(2n)))$, see~\cite[Section~11.4]{mcc2001}. The kernel of $\pi^*$ is given by the images of the differentials $d_r$ of this spectral sequence in its bottom row.

Note that the action of $G$ on $O(2n)$ is induced by the inclusion $G\in O(2n)$, and the cohomology of $O(2n)$ is acted on by $G$ through its factor group $G/G^+$ of order $2$. Here $G^+$ denotes the elements of $G$ with positive determinant. Hence we can replace $G$ by $G^+$ and simultaneously  pass from the sheaf $\mathcal H^q(O(2n))$ to the cohomology $H^q(SO(2n))$ (see~\cite{bro1982}, Ch.~III, Proposition~6.2), thus obtaining
$$
E^{p,q}_2 = H^p(BG, \mathcal H^q(O(2n))) = H^p(BG^+, H^q(SO(2n))).
$$
In order to find the images of $d_r$'s, note that the fiber bundle 
$$
\begin{CD}
SO(2n) @>>> SO(2n)\times_G EG^+\\
@. @VVV\\
@. BG^+
\end{CD}
$$
is induced from the fiber bundle 
$$
\begin{CD}
SO(2n) @>>> ESO(2n)\\
@. @VVV\\
@. BSO(2n)
\end{CD}
$$
by the inclusion $G^+\to SO(2n)$. In the spectral sequence of the latter fiber bundle all the primitive generators of $H^*(SO(2n))$ are transgressive. They are mapped to the bottom row by the corresponding differentials $d_r$, their images being the Stiefel-Whitney classes of $O(2n)$. Thus, in the considered spectral sequence, the differentials $d_r$ are generated by the transgressions that send the primitive generators of $H^*(SO(2n))$ to the Stiefel-Whitney classes of the representation of $G^+$ on $\mathbb R^{2n}$. Denote this representation $W_{2n}$.

Let us summarize as follows. 

\begin{lem}
\label{ideal-ind}
The kernel of the natural map $\pi^* : H^*(BG)\to H^*(O(2n)/G)$ is generated by the homogeneous components of positive degree of the expression
$$
\sum_{0\le i < j \le n} w_i\odot w_j + \sum_{i=0}^n (1+c)^{n-i} \Sqe w_i.
$$
\end{lem}

\begin{proof}
In the bottom row of the spectral sequence passing from $H^*(BG)$ to $H^*(BG^+)$ ``kills'' the element $w_1(W_{2n})$ and the ideal generated by it. The other differentials ``kill'' the other classes $w_r(W_{2n})$ by the above considerations.

It remains to calculate the Stiefel-Whitney classes of $W_{2n}$. Remind that by the Stiefel-Whitney classes of a representation we mean the Stiefel-Whitney classes of the vector bundle $\eta : (W_{2n}\times EG)/G\to BG$. Denote $V_n$ the natural representation of $O(n)$, and consider its corresponding bundle $\xi : (V_n\times EO(n))/O(n) \to BO(n)$. It can be checked by definition that $\eta=\Sqe\xi$ and the claim follows by applying Lemma~\ref{sw-sq}.
\end{proof}

Now the proof of Lemma~\ref{index-odd} is finished as follows: we have to find the nilpotency degree of $c$ in $H^*(BG)/\ker\pi^*$. If $n$ is odd, then the one-dimensional generator of $\ker \pi^*$ is 
$$
c + w_1\odot 1,
$$
hence $c\not=0$, $c^2=0$ by Lemma~\ref{odot-ann}, and $\hind G_{2n}^n=1$ in this case. 

If $n=2\mod 4$, then we have the relations in dimensions $2$ and $3$
\begin{eqnarray*}
c^2+\Sqe w_1 + 1\odot w_2 &=& 0\\
c\Sqe w_1 + 1\odot w_3 + w_1\odot w_2 = 0.
\end{eqnarray*}
Substituting $\Sqe w_1 = c^2 + 1\odot w_2$ from the first relation to the second we obtain
$$
c^3 = 1\odot w_3 + w_1 \odot w_2,
$$
hence $c^4 = 0$ by Lemma~\ref{odot-ann}, and $\hind G_{2n}^n=3$ in this case.

Now let us turn to Lemma~\ref{index-2l}. Let $n=2^l$, and let us add the additional relations of the form $w_i = 0$ for all $i$ except $i= 2^l-2^k$ ($k=0,\ldots, l$) and $i=2^l$. In this case the remaining relations in $\ker\pi^*$ are 
\begin{eqnarray*}
c^{2^l} &=& \Sqe w_{2^l-2^{l-1}} + 1\odot w_{2^l}\\
c^{2^{l-1}} \Sqe w_{2^l-2^{l-1}} &=& \Sqe w_{2^l-2^{l-2}} + w_{2^l-2^{l-1}}\odot w_{2^l}\\
&\dots&\\
c^2 \Sqe w_{2^l-2} &=& \Sqe w_{2^l-1} + w_{2^l-2}\odot w_{2^l}\\
c \Sqe w_{2^l-1} &=& w_{2^l-1}\odot w_{2^l}\\
\Sqe w_{2^l} &=& 0,
\end{eqnarray*}
along with the relations of the form
$$
w_{2^l-2^k}\odot w_{2^l-2^m} = 0,\ 0\le k<m\le l.
$$
Thus we obtain $c^{2^{l+1}-1} = c^{2n-1} = w_{2^l}\odot w_{2^l-1}\not=0$. Also, we must have $c^{2n}=0$ by the upper bound $\hind G_{2n}^n\le 2n-1$, without any additional relations. Therefore, $\hind G_{2n}^n=2n-1$ in this case.

\end{document}